\documentclass[12pt]{article}
\usepackage{amssymb}
\usepackage{amsmath}
\usepackage{amsthm}
\usepackage{color}
\usepackage{comment}
\usepackage{geometry}		
\usepackage{graphicx}

\renewcommand{\labelenumi}{\roman{enumi})}

\let\originalleft\left
\let\originalright\right
\renewcommand{\left}{\mathopen{}\mathclose\bgroup\originalleft}
\renewcommand{\right}{\aftergroup\egroup\originalright}

\begin{document}

\newcommand{\bO}{{\bf 0}}
\newcommand{\ee}{{\varepsilon}}
\newcommand{\rD}{{\rm D}}

\newtheorem{theorem}{Theorem}[section]
\newtheorem{corollary}[theorem]{Corollary}
\newtheorem{lemma}[theorem]{Lemma}
\newtheorem{proposition}[theorem]{Proposition}
\newtheorem{algorithm}[theorem]{Algorithm}

\theoremstyle{definition}
\newtheorem{definition}{Definition}[section]
\newtheorem{example}[definition]{Example}

\theoremstyle{remark}
\newtheorem{remark}{Remark}[section]

\title{
Chaos in the border-collision normal form: A computer-assisted proof using induced maps and invariant expanding cones
}
\author{
P.A.~Glendinning$^{\dagger}$ and D.J.W.~Simpson$^{\ddagger}$\\
\small $^{\dagger}$Department of Mathematics, University of Manchester, Oxford Road, Manchester M13 9PL, U.K.\\
\small $^{\ddagger}$School of Fundamental Sciences, Massey University, Palmerston North 4410, New Zealand
}
\maketitle


\begin{abstract}

In some maps the existence of an attractor with a positive Lyapunov exponent can be proved by constructing a trapping region in phase space and an invariant expanding cone in tangent space. If this approach fails it may be possible to adapt the strategy by considering an induced map (a first return map for a well-chosen subset of phase space). In this paper we show that such a construction can be applied to the two-dimensional border-collision normal form (a continuous piecewise-linear map) if a certain set of conditions are satisfied and develop an algorithm for checking these conditions. The algorithm requires relatively few computations, so it is a more efficient method than, for example, estimating the Lyapunov exponent from a single orbit in terms of speed, numerical accuracy, and rigor. The algorithm is used to prove the existence of an attractor with a positive Lyapunov exponent numerically in an area of parameter space where the map has strong rotational characteristics and the consideration of an induced map is critical for the proof of robust chaos.

\end{abstract}

\section{Introduction}
\label{sec:intro}
\setcounter{equation}{0}

Piecewise-smooth dynamical systems 
have different evolution rules 
in different parts of phase space.
They provide natural mathematical models for engineering applications
involving impacts or on-off control strategies \cite{AwLa03,Jo03},
are useful for understanding biological systems including gene switching \cite{EdGl14},
and have been employed in computer science, particularly cryptography \cite{KoLi11}.
From a theoretical viewpoint, piecewise-linear systems are commonly used as a test-bed for understanding nonlinear dynamics
as they are reasonably amenable to an exact analysis,
an example being the Lozi map \cite{Lo78} as a piecewise-linear version of the H\'enon map.

Although the ideas presented in this paper are more general, we use
the two-dimensional border-collision normal form (2d BCNF) --- a normal form for continuous maps on $\mathbb{R}^2$
comprised of two affine pieces, as our canonical example.
The 2d BCNF is the family of difference equations with $(x^\prime,y^\prime)^T=f(x,y)$ where
\begin{equation}
f(x,y) = \begin{cases}
A_L \begin{bmatrix} x \\[-.8mm] y \end{bmatrix} + \begin{bmatrix} \mu \\[-.8mm] 0 \end{bmatrix}, & x \le 0, \\[2mm]
A_R \begin{bmatrix} x \\[-.8mm] y \end{bmatrix} + \begin{bmatrix} \mu \\[-.8mm] 0 \end{bmatrix}, & x \ge 0,
\end{cases}
\label{eq:f}
\end{equation}
and with
\begin{align}
A_L &= \begin{bmatrix} \tau_L & 1 \\ -\delta_L & 0 \end{bmatrix}, &
A_R &= \begin{bmatrix} \tau_R & 1 \\ -\delta_R & 0 \end{bmatrix}.
\nonumber
\end{align}
The 2d BCNF has been widely studied, see for instance \cite{DiBu08,GlJe19,Si16} and references within.
In this paper we restrict our attention to parameters with
\begin{align}
\tau_R &\in \mathbb{R}, &
\tau_L &> 0, &
\delta_L, \delta_R &> 0, &
\mu &= 1,
\label{eq:params}
\end{align}
with which $f$ is invertible and orientation-preserving.
The role of $\mu$ is to control the border-collision bifurcation:
the 2d BCNF was originally derived in \cite{NuYo92}
as the leading order terms of a map for which a fixed point collides with a switching manifold when $\mu = 0$.
In view of a linear rescaling it is sufficient to consider $\mu \in \{-1,0,1\}$ and we have put $\mu = 1$. 
The condition $\tau_L>0$ is needed for the definition of the induced map in \S\ref{sec:trappingRegion}.
If $\tau_L = -\tau_R$ and $\delta_L = \delta_R$ then the 2d BCNF reduces to the Lozi map. 

In the seminal paper \cite{BaYo98}, Banerjee {\em et.~al.}~showed that the 2d BCNF is
relevant for describing the behaviour of power converters
and pointed out that strange attractors could exist over open sets of parameter values,
a phenomenon they called {\em robust chaos}.
A more recent formulation and rigorous proof of their insights can be found in \cite{Gl17,GlSi21}.
For the Lozi map such robust chaos had been established much earlier by Misiurewicz \cite{Mi80}.

Whilst the proof of \cite{GlSi21} establishes robust chaos
in the 2d BCNF for the parameter constraints described in \cite{BaYo98},
it is clear from numerical simulations that these constraints are not optimal.
The aim of this paper is to obtain implicit conditions for the existence of chaotic attractors which,
whilst well-nigh impossible to verify by hand, are relatively easy to verify numerically
and allow us to demonstrate (up to computer accuracy and over a discretised parameter grid)
that the 2d BCNF has a chaotic attractor over larger regions of parameter space.
This approach does not rely on the accurate simulation of individual orbits
so is more reliable than an analysis based on a large number of points of one orbit
where rounding errors can lead to misleading results.

The key tool used in this paper is an {\em induced map}.
An induced map $F$ is essentially a first return map for a particular subset of phase space.
That is, for any point $Z$ in this set, $F(Z) = f^n(Z)$ where the number of iterations $n$ is $Z$-dependent.
Induced maps are heavily employed in the study of one-dimensional maps \cite{DeVa93} and an application to the BCNF
is given in \cite{Gl16e}.
In this paper we construct a trapping region and an invariant expanding cone for $F$
in order to establish the existence of robust chaos for $f$.

The sections of this paper are organised to follow the steps of the construction.
First in \S\ref{sec:induced} we define and characterise an induced map $F$ for a well-chosen subset of phase space.
Then in \S\ref{sec:Dp} and \S\ref{sec:trappingRegion} we derive conditions
for $F$ to be well-defined and to have a forward invariant region $\Omega$.
Although $F$ is not continuous and $F(\Omega)$ is not contained in the interior of $\Omega$
(as is necessary for $\Omega$ to be a trapping region and 
which implies the existence of an attractor in $\Omega$),
these issues can be circumvented by imposing a cylindrical topology on $\Omega$.
 
In \S\ref{sec:Lyap} we define invariant expanding cones
and show how their existence implies a positive Lyapunov exponent.
Then in \S\ref{sec:GH}, \S\ref{sec:G}, and \S\ref{sec:iecExistence}
we derive conditions for the existence of such a cone.
The conditions are not given explicitly in terms of the parameters \eqref{eq:params}
but are based on the roots of quadratic polynomials so are elementary to check numerically.
In \S\ref{sec:pMin} we explain why these conditions can we expected to hold when an additional constraint is placed on the parameters.

Then in \S\ref{sec:algorithm} we collate the conditions into an
algorithm which determines up to numerical accuracy whether or not all conditions hold for a given set of parameter values.
This is illustrated for a two-dimensional slice of parameter space in \S\ref{sec:numerics}.
Finally in \S\ref{sec:conc} we discuss generalisations and future directions.

\section{The induced map}
\label{sec:induced}
\setcounter{equation}{0}

The switching manifold of \eqref{eq:f} is the $y$-axis
denoted by $\Sigma$.
For all values of the parameters, $f(\Sigma)$ is the $x$-axis.
For parameters satisfying \eqref{eq:params}, the map $(x^\prime ,y^\prime )^T=f(x,y)$ has the property that
\begin{equation}
\text{the sign of $y'$ is opposite to the sign of $x$}.
\label{eq:signs}
\end{equation}
It follows immediately from this observation that if $Q_i$ denotes the closure of the $i^{\rm th}$ quadrant of $\mathbb{R}^2$ then
\begin{equation}
\begin{split}
f(Q_1), f(Q_4) &\subset Q_3 \cup Q_4 \,,\\
f(Q_2), f(Q_3) &\subset Q_1 \cup Q_2 \,,
\end{split}
\label{eq:quadrants}
\end{equation}
as shown in Fig.~\ref{fig:schemQuadrants}.
Let $\Phi$ be the set obtained by removing $\Sigma$ from $Q_3$, i.e.
\begin{equation}
\Phi = \left\{ (x,y) \,\big|\, x < 0, y \le 0 \right\}.
\label{eq:Fdomain}
\end{equation}
For a given map $f$ of the form \eqref{eq:f} with \eqref{eq:params}, let
$\Phi_{\rm pre} = \bigcup_{n=1}^\infty f^{-n}(\Phi)$
be all points that eventually map to $\Phi$, and let
\begin{equation}
\Phi_0 = \Phi_{\rm pre} \cap \Phi.
\label{eq:Phi0}
\end{equation}

\begin{figure}[b!]
\begin{center}
\includegraphics[height=4.2cm]{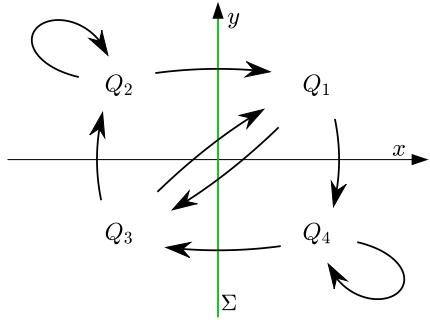}
\caption{
An illustration of the action of \eqref{eq:f} with \eqref{eq:params} on the quadrants of $\mathbb{R}^2$.
For example the image of any point in $Q_1$ belongs to either $Q_3$ or $Q_4$.
Take care to note that while this figure has rotational symmetry,
the map does not have rotational symmetry (the origin $(0,0)$ maps to $(1,0)$).
\label{fig:schemQuadrants}
}
\end{center}
\end{figure}

\begin{definition}
The {\em induced map} $F : \Phi_0 \to \Phi$ is defined as
\begin{equation}
F(Z) = f^n(Z),~
\text{for the smallest $n \ge 1$ for which $f^n(Z) \in \Phi$}.
\label{eq:F}
\end{equation}
\label{df:F}
\end{definition}

\begin{figure}[b!]
\begin{center}
\includegraphics[height=5cm]{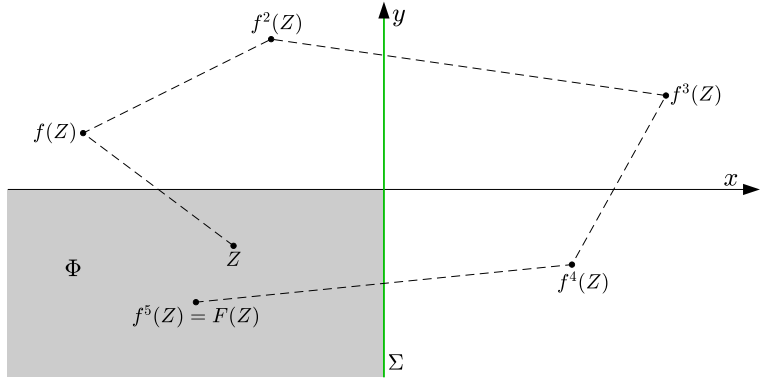}
\caption{
A sketch of part of the forward orbit of a point $Z \in \Phi_0 \subseteq \Phi$
illustrating the induced map $F(Z)$.
\label{fig:schemQQa}
}
\end{center}
\end{figure}

Fig.~\ref{fig:schemQQa} illustrates the construction of $F$.
For the orbit shown we have $F(Z) = f_R^2 \left( f_L^3(Z) \right)$, where
\begin{align}
f_L(Z) &= A_L Z + \begin{bmatrix} \mu \\ 0 \end{bmatrix}, &
f_R(Z) &= A_R Z + \begin{bmatrix} \mu \\ 0 \end{bmatrix},
\nonumber
\end{align}
denote the components of $f$.
Given the way $f$ maps the quadrants of $\mathbb{R}^2$ illustrated in Fig.~\ref{fig:schemQuadrants},
for any $Z \in \Phi_0$ there exist $p, q \ge 1$ such that
\begin{equation}
F(Z) = f_R^q \left( f_L^p(Z) \right).
\label{eq:Fpq}
\end{equation}
In the remainder of this section we prove this assertion.
Let
\begin{equation}
\begin{split}
\Pi_L &= \left\{ (x,y) \,\big|\, x \le 0, y \in \mathbb{R} \right\}, \\
\Pi_R &= \left\{ (x,y) \,\big|\, x \ge 0, y \in \mathbb{R} \right\},
\end{split}
\end{equation}
denote the closed left and right half-planes.

\begin{definition}
Given $Z \in \mathbb{R}^2$,
let $\chi_L(Z)$ be the smallest $i \ge 1$ for which $f^i(Z) \notin \Pi_L$
and let $\chi_R(Z)$ be the smallest $j \ge 1$ for which $f^j(Z) \notin \Pi_R$,
if such $i$ and $j$ exist.
\label{df:chi}
\end{definition}

\begin{lemma}
Let $Z \in \Phi_0$.
Then $p = \chi_L(Z)$ and $q = \chi_R \left( f_L^p(Z) \right)$ exist
and $F(Z)$ is given by \eqref{eq:Fpq}.
\label{le:Fpq}
\end{lemma}

\begin{proof}
By assumption the forward orbit of $Z$ under $f$ returns to $\Phi$.
To do so the orbit must first enter $\Pi_R$ because $f^{-1}(\Phi) \subset \Pi_R$ by \eqref{eq:signs}.
Let $i \ge 1$ be the smallest number for which $f^i(Z) \in \Pi_R$.
If $f^i(Z) \in {\rm int}(\Pi_R)$ then $f^i(Z) \notin \Pi_L$ and so $p = \chi_L(Z) = i$ and $f^i(Z) = f_L^p(Z)$,
otherwise $f^i(Z)$ lies on the positive $y$-axis in which case $f^{i+1}(Z)$ lies on the positive $x$-axis
and so $p = \chi_L(Z) = i+1$ and $f^{i+1}(Z) = f_L^p(Z)$.
The number $q = \chi_R \left( f_L^p(Z) \right)$ exists because the forward orbit of $Z$ returns to $\Phi$.
Moreover, $f^{p+q}(Z) = f_R^q \left( f_L^p(Z) \right)$
and $f^{p+q}(Z) \in \Phi$ because its first component is negative (by the definition of $\chi_R$)
and its second component is non-positive (because the first component of $f^{p+q-1}(Z)$ is non-negative by the definition of $\chi_R$),
thus $f^{p+q}(Z) = F(Z)$.
\end{proof}

\section{Dividing phase space by preimages of the switching manifold}
\label{sec:Dp}
\setcounter{equation}{0}

In this section we address the dynamics of \eqref{eq:f} in the left half-plane $\Pi_L$.
Since $f_L$ is invertible and affine, $f_L^{-i}(\Sigma)$ is a line for all $i \ge 1$.
If this line is not vertical we write
\begin{equation}
f_L^{-i}(\Sigma) = \left\{ (x,y) \in \mathbb{R}^2 \,\big|\, y = m_i x + c_i \right\},
\label{eq:swManPreimages}
\end{equation}
where $m_i$ and $c_i$ are its slope and $y$-intercept.
Let $p^*$ be the smallest value of $i \ge 1$ for which $m_i \ge 0$,
with $p^* = \infty$ if $m_i < 0$ for all $i \ge 1$.
Let $\phi = \cos^{-1} \!\big( \frac{\tau_L}{2 \sqrt{\delta_L}} \big)\in \left( 0, \frac{\pi}{2} \right)$.
As shown in \cite{Si20e},
\begin{equation}
p^* = \begin{cases}
\left\lceil \frac{\pi}{\phi} - 1 \right\rceil, & 0 < \tau_L < 2 \sqrt{\delta_L} \,, \\
\infty, & \tau_L \ge 2 \sqrt{\delta_L} \,.
\end{cases}
\label{eq:pStar}
\end{equation}
Moreover, for $i = 1,\ldots,p^*$
the $y$-intercepts $c_i$ form a decreasing sequence
whereas the slopes $m_i$ form an increasing sequence \cite{Si20e}.
Consequently the regions
\begin{equation}
\begin{split}
D_1 &= \left\{ (x,y) \in \Pi_L \,\big|\, y > m_1 x + c_1 \right\}, \\
D_p &= \left\{ (x,y) \in \Pi_L \,\big|\, m_p x + c_p < y \le m_{p-1} x + c_{p-1} \right\}, \qquad
\text{for}~ p = 2,\ldots,p^*,
\end{split}
\label{eq:Dp}
\end{equation}
are disjoint and partition $\Pi_L$ above $f_L^{-p^*}(\Sigma)$, as shown in Fig.~\ref{fig:schemQQc}.
Under $f$ every point in $D_1$ maps to the interior of $\Pi_R$,
while for any $p \in \{ 2,\ldots,p^* \}$ every point in $D_p$ maps to $D_{p-1}$ \cite{Si20e}.
Consequently we have the following relationship between $D_p$ and $\chi_L$.

\begin{figure}[b!]
\begin{center}
\includegraphics[height=5cm]{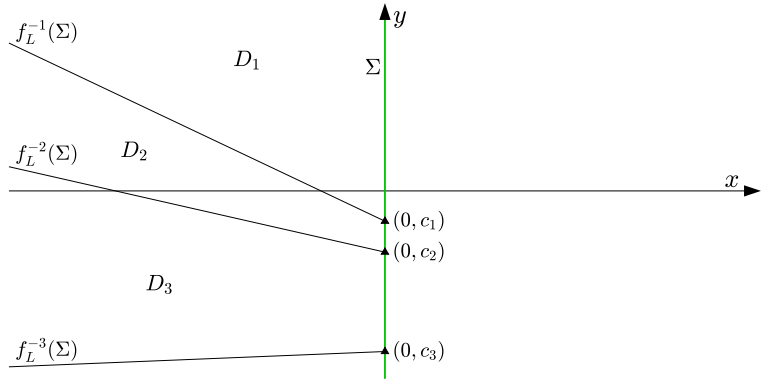}
\caption{
A sketch showing the regions $D_1,\ldots,D_{p^*}$ \eqref{eq:Dp}
in a case for which $p^* = 3$.
\label{fig:schemQQc}
}
\end{center}
\end{figure}

\begin{lemma}
If $Z \in D_p$ then $\chi_L(Z) = p$.
\label{le:chiL}
\end{lemma}

\section{A trapping region for the induced map.}
\label{sec:trappingRegion}
\setcounter{equation}{0}

Notice $p^*\ge 2$ by \eqref{eq:pStar}.
Furthermore, the preimage of $\Sigma$ in $\Pi_L$ consists of points with $\tau_L x + y + 1 = 0$ and so $m_1 = -\tau_L$ and $c_1 = -1$.
Given $1 \le p_{\rm min} < p_{\rm max} \le p^*$ (with $p_{\rm max}$ finite), let
\begin{align}
S &= \left( 0, c_{p_{\rm max}} \right), &
T &= \left( 0, c_{p_{\rm min}} \right).
\label{eq:st}
\end{align}
Observe that $f(S) = \left( c_{p_{\rm max}} + 1, 0 \right)$ and $f(T) = \left( c_{p_{\rm min}} + 1, 0 \right)$
are points on the $x$-axis with $x \le 0$ (because $c_{p_{\rm min}} \le c_1 = -1$).
Also $f^{p_{\rm max}}(S) \in \Sigma$ and $f^{p_{\rm min}}(T) \in \Sigma$, so
\begin{equation}
\chi_L(S) = p_{\rm max} + 1, \quad \chi_L(T) = p_{\rm min} + 1.
\label{eq:chiLST}
\end{equation}
Let
\begin{equation}
\Omega = \left\{ (x,y) \in \Phi \,\Big|\,
\tfrac{c_{p_{\rm max}}+1}{c_{p_{\rm max}}}(c_{p_{\rm max}}-y) \le x \le
\tfrac{c_{p_{\rm min}}+1}{c_{p_{\rm min}}}(c_{p_{\rm min}}-y) \right\},
\label{eq:Omega}
\end{equation}
be the quadrilateral with vertices $S$, $T$, $f(S)$, and $f(T)$, as shown in Fig.~\ref{fig:schemQQe}.
Assume
\begin{align}
f(S), f(T) \in \Phi_{\rm pre} \,,
\label{eq:fSfT}
\end{align}
so that $q_S = \chi_R \left( f_L^{\chi_L(S)}(S) \right)$ and $q_T = \chi_R \left( f_L^{\chi_L(T)}(T) \right)$ are well-defined,
and let
\begin{align}
q_{\rm min} &= \min \left[ q_S, q_T \right], &
q_{\rm max} &= \max \left[ q_S, q_T \right] + 1.
\label{eq:qMinqMax}
\end{align}
As illustrated in Fig.~\ref{fig:schemQQe},
let $U$ denote the point at which the line through $F(S)$ and $f^{-1}(F(S))$ intersects $\Sigma$,
and similarly let $V$ denote the point at which the line through $F(T)$ and $f^{-1}(F(T))$ intersects $\Sigma$.

From Fig.~\ref{fig:schemQQe} it is intuitively clear that the desired property
$F(\Omega) \subseteq \Omega$ will require a number of conditions on the points $U$, $V$, $F(S)$, and $F(T)$.
So that these conditions can be expressed in a way that a computer can check,
we let $\pi_i:\mathbb{R}^2\to \mathbb{R}$, for $i = 1,2$, be the standard projections onto the axes,
$\pi_1(x,y)=x$ and $\pi_2(x,y)=y$.

\begin{figure}[b!]
\begin{center}
\includegraphics[height=5cm]{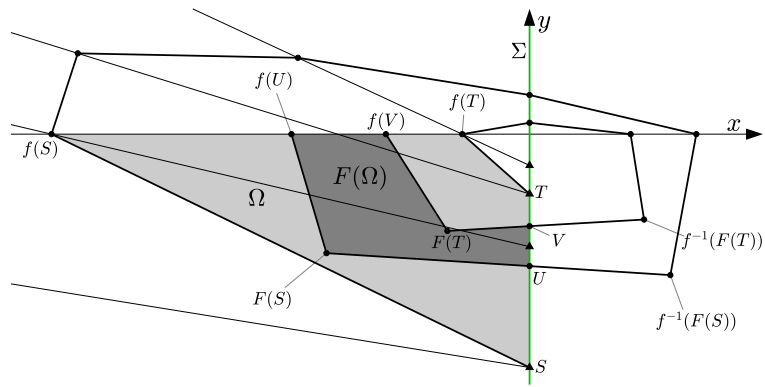}
\caption{
An example of $\Omega$ satisfying the conditions of Proposition \ref{pr:OmegaForwardInvariant}.
Here $p_{\rm min} = 2$, $p_{\rm max} = 4$, $q_{\rm min} = 2$, and $q_{\rm max} = 3$.
By iterating the line segments $S f(S)$ and $T f(T)$ under $f$ we are able to form $F(\Omega)$ as shown.
To produce this figure we used $(\tau_L,\delta_L,\tau_R,\delta_R) = (1.1,0.4,0.4,2)$.
\label{fig:schemQQe}
}
\end{center}
\end{figure}

\begin{proposition}
Suppose \eqref{eq:fSfT} is satisfied and
\begin{align}
& \pi_2(U)>\pi_2(S), \label{eq:Ucond} \\
& \pi_2(V)<\pi_2(T), \label{eq:Vcond} \\
& \pi_1(f(S))\pi_2(F(S))+\pi_2(S)\pi_1(F(S))<\pi_1(f(S))\pi_2(S), ~\textrm{and} \label{eq:FScond} \\
& \pi_1(f(T))\pi_2(F(T))+\pi_2(T)\pi_1(F(T))>\pi_1(f(T))\pi_2(T). \label{eq:FTcond}
\end{align}
Then $\Omega \subseteq \Phi_0$,
$F(\Omega) \subseteq \Omega$, and
\begin{align}
p_{\rm min} &\le \chi_L(Z) \le p_{\rm max} \,, \label{eq:chiL} \\
q_{\rm min} &\le \chi_R \left( f_L^{\chi_L(Z)}(Z) \right) \le q_{\rm max} \,, \label{eq:chiR}
\end{align}
for all $Z \in \Omega$.
\label{pr:OmegaForwardInvariant}
\end{proposition}

Conditions \eqref{eq:Ucond}--\eqref{eq:FTcond} can be interpreted more intuitively as
\begin{align}
& \text{$U$ lies above $S$,} \label{eq:Ucondint} \\
& \text{$V$ lies below $T$,} \label{eq:Vcondint} \\
& \text{$F(S)$ lies to the right of the line through $S$ and $f(S)$, and} \label{eq:FScondint} \\
& \text{$F(T)$ lies to the left of the line through $T$ and $f(T)$,} \label{eq:FTcondint}
\end{align}
respectively.
These conditions are all satisfied in Fig.~\ref{fig:schemQQe}.

\begin{proof}[Proof of Proposition \ref{pr:OmegaForwardInvariant}]
By construction, ${\displaystyle \Omega \subset \bigcup_{p = p_{\rm min}}^{p_{\rm max}} D_p}$
(note $S \notin \Omega$ whereas $f(T) \in D_{p_{\rm min}}$).
This verifies \eqref{eq:chiL} by Lemma \ref{le:chiL}.

Let $\Psi_L$ be the (compact filled) polygon formed by connecting (in order) the points
\begin{equation}
S, f(S), \ldots, f^{p_{\rm max}}(S), f^{p_{\rm min}}(T), \ldots, f(T), T,
\nonumber
\end{equation}
as shown in Fig.~\ref{fig:schemQQf}.
These points all belong to $\Pi_L$, thus $\Psi_L \subset \Pi_L$.
It is a simple exercise to show that $\Psi_L$ is simple (i.e.~its boundary has no self-intersections).
Then $f(\Psi_L) = f_L(\Psi_L)$ and since $f_L$ is affine
$f(\Psi_L)$ is the polygon with vertices
\begin{equation}
f(S), f^2(S) \ldots, f^{p_{\rm max}+1}(S), f^{p_{\rm min}+1}(T), \ldots, f^2(T), f(T).
\nonumber
\end{equation}
Observe $\Omega \subset \Psi_L$ and $f(\Psi_L) \cap \Pi_L \subset \Psi_L$.
Thus for any $Z \in \Omega$ the forward orbit of $Z$ under $f$
remains in $\Psi_L$ until escaping $\Pi_L$ by arriving at $f_L^{\chi_L(Z)}(Z)$.
Moreover $f_L^{\chi_L(Z)}(Z)$ belongs to the quadrilateral (or triangle if $p_{\rm min} = 1$) with vertices
\begin{equation}
f^{p_{\rm max}}(S), f^{p_{\rm max}+1}(S), f^{p_{\rm min}+1}(T), f^{p_{\rm min}}(T),
\nonumber
\end{equation}
call this region $\Delta$.

\begin{figure}[b!]
\begin{center}
\includegraphics[height=5cm]{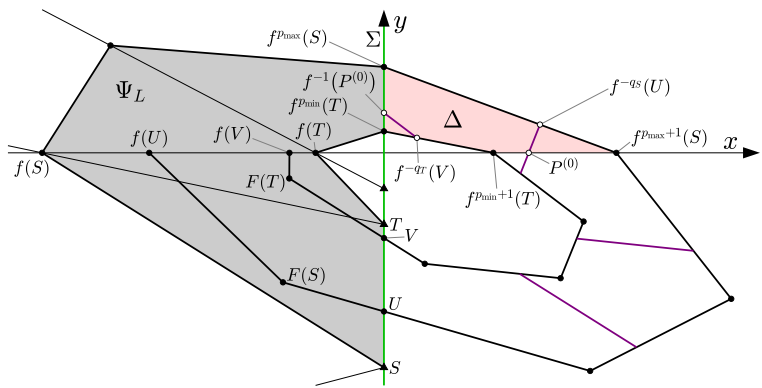}
\caption{
Elements introduced in the proof of Proposition \ref{pr:OmegaForwardInvariant}.
Here $p_{\rm min} = 2$, $p_{\rm max} = 3$, $q_S = 3$, and $q_T = 4$ (so $q_{\rm min} = 3$, and $q_{\rm max} = 5$).
The grey region is $\Psi_L$; the pink region is $\Delta$.
Parts of preimages of $\Sigma$ under $f_R$ are shown as purple line segments.
Two of these preimages divide $\Delta$ into three regions:
in right-most region $\chi_R = 3$,
in the middle region $\chi_R = 4$, and
in the left-most region $\chi_R = 5$.
To produce this figure we used $(\tau_L,\delta_L,\tau_R,\delta_R) = (1,0.6,1.2,1.2)$.
\label{fig:schemQQf}
}
\end{center}
\end{figure}

Now let $\Psi_R$ be the polygon formed by connecting the points
\begin{equation}
f^{p_{\rm max}}(S), f^{p_{\rm max}+1}(S), \ldots, f^{p_{\rm max}+q_S-1}(S), U,
V, f^{p_{\rm min}+q_T-1}(T), \ldots, f^{p_{\rm min}+1}(T), f^{p_{\rm min}}(T).
\nonumber
\end{equation}
The polygon $\Psi_R$ is simple and belongs to $\Pi_R$.
Thus $f(\Psi_R) = f_R(\Psi_R)$ is the polygon with vertices
\begin{equation}
f^{p_{\rm max}+1}(S), f^{p_{\rm max}+2}(S), \ldots, f^{p_{\rm max}+q_S}(S), f(U),
f(V), f^{p_{\rm min}+q_T}(T), \ldots, f^{p_{\rm min}+2}(T), f^{p_{\rm min}+1}(T).
\nonumber
\end{equation}
Since $\Delta \subset \Psi_R$ and $f(\Psi_R) \cap \Pi_R \subset \Psi_R$,
the forward orbit of $f_L^{\chi_L(Z)}(Z) \in \Delta$
either fails to escape $\Pi_R$ (which below we show is not possible)
or escapes $\Pi_R$ by arriving at $F(Z)$ in the polygon with vertices
\begin{equation}
U, f^{p_{\rm max}+q_S}(S), f(U), f(V), f^{p_{\rm min}+q_T}(T), V.
\label{eq:FOmega}
\end{equation}
By \eqref{eq:Ucond}--\eqref{eq:FTcond} this implies $F(Z) \in \Omega$.
Therefore, once we establish \eqref{eq:chiR}, which we do next,
we have $\Omega \subseteq \Phi_0$ and $F(\Omega) \subseteq \Omega$.

We now use preimages of $\Sigma$ under $f_R$ to partition $\Delta$ into regions of constant $\chi_R$.
For ease of explanation suppose $q_S < q_T$ as in Fig.~\ref{fig:schemQQf} (the proof can be completed similarly if $q_S \ge q_T$).
The line $f_R^{-q_S}(\Sigma)$ contains the points $f^{-q_S}(U)$ (which lies on $\partial \Delta$ --- the boundary of $\Delta$)
and $f^{-q_S}(V)$ (which lies below the $x$-axis),
thus $f_R^{-q_S}(\Sigma)$ intersects $\partial \Delta$ at $f^{-q_S}(U)$
and some point $P^{(0)}$ on the edge connecting $f^{p_{\rm max}+1}(S)$ and $f^{p_{\rm min}+1}(T)$.
Similarly if $q_S < q_T - 1$ then for each $j = 1,2,\ldots,q_T-q_S-1$,
there exists a point $P^{(j)}$ on this edge and to the left of $P^{(j-1)}$ such that
$f_R^{-(q_S+j)}(\Sigma)$ intersects $\partial \Delta$
at $f^{-1} \left( P^{(j-1)} \right)$ and $P^{(j)}$.
Also $f_R^{-q_T}(\Sigma)$ intersects $\partial \Delta$
at $f^{-1} \left( P^{(q_T-q_S-1)} \right)$ and $f^{-q_T}(V)$.

This shows that these preimages of $\Sigma$ have no intersections within $\Delta$
and partition $\Delta$ into $q_T - q_S + 2$ polygonal regions.
Call these regions $\Delta_q$, for $q = q_S,q_S+1,\ldots,q_T + 1$,
where $\Delta_q \subset \Delta$ has boundaries $f_R^{-(q-1)}(\Sigma)$ (unless $q = q_S$) and $f_R^{-q}(\Sigma)$ (unless $q = q_T + 1$),
and where $\Delta_q$ includes the first boundary but not the second.
Now choose any $W \in \Delta \setminus \Sigma$ and let $q$ be such that $W \in \Delta_q$.
By construction, $f^i(W) \in \Pi_R$ for all $i = 0,1,\ldots,q-1$ and $f^q(W) \in \Pi_R$.
Thus $\chi_R(W) = q$ and this verifies \eqref{eq:chiR}.
\end{proof}

Although $F$ is not continuous and $\Omega$ is not a trapping region (because it does not map to its interior)
these deficiencies can be removed as follows.
For all $c_{p_{\rm max}} \le y \le c_{p_{\rm min}}$,
we {\em identify} the point $(0,y)$ with its image $f(0,y)$
to endow $\Omega$ with a cylindrical topology.

\begin{proposition}
With the same assumptions as Proposition \ref{pr:OmegaForwardInvariant},
in the cylindrical topology $F$ is continuous on $\Omega$
and $F(\Omega) \subset {\rm int}(\Omega)$.
\label{pr:FContinuous}
\end{proposition}

Before giving the proof of Proposition~\ref{pr:FContinuous} it is worth sketching where problems arise and why the 
cylindrical topology is necessary. Consider two points, $X$ and $Y$ in $\Omega$ with $X\in D_p$ and $Y\in D_{p+1}$ 
and with $X$ close to $Y$. Then $f^p(X)=f_L^p(X)\in \Pi_R$ but $f^p(Y)=f_L^p(Y)\in \Pi_L$. Even so, 
$f^p(X)$ and $f^p(Y)$ are close by continuity of $f_L$ and then the continuity of $f$ implies that $f^{p+1}(X)$ 
and $f^{p+1}(Y)$ are also close. Thus the transition from $\Pi_L$ to $\Pi_R$
does not create problems for the continuity of $F$.

Suppose now that $f^n(X)$ and $f^n(Y)$ are close but on different sides of $\Sigma$ with $n>p$,
i.e.~$f^n(X)=F(X)\in \Omega$ but $f^n(Y)\in \Pi_R$. Then $f^{n+1}(Y)=F(Y)\in \Omega$
and (by continuity across $\Sigma$) is close to 
$f^{n+1}(X)$, which is not in $\Omega$ (it is in $Q_2$). However, in the cylindrical topology obtained by identifying
$\Sigma \cap {\rm cl}(\Omega )$ with its image under $f$, which is on the $x$-axis, since $f^{n+1}(X)$ and $f^{n+1}(Y)$ 
are close in the Euclidean toplogy then so is $F(X)$ and $F(Y)$ in the induced cylindrical topology.


\begin{proof}[Proof of Proposition \ref{pr:FContinuous}]
Choose any $Z \in \Omega$ and let $n \ge 1$ be as in Definition \ref{df:F}.
Let $Y \in \Omega$ be a point close to $Z$ in the cylindrical topology.
Under $f$ the forward orbit of $Y$ is close to the forward orbit of $Z$
because in the cylindrical topology points in $\Sigma$ are identified with their images under $f$.
Therefore $F(Y)$ is near $F(Z)$, establishing continuity,
unless $f^i(Z) \in \partial \Omega$ (the boundary of $\Omega$) for some $i \in \{ 1,\ldots,n-1 \}$,
in which case the forward orbit of $Y$ may return to $\Omega$ prematurely.
If $f^{n-1}(Z) \in \partial \Omega$ then $f^{n-1}(Z) \in \Sigma$
so again $F(Y)$ is near $F(Z)$ by the definition of the cylindrical topology.
But $f^i(Z) \in \partial \Omega$ for some $i \in \{ 1,\ldots,n-2 \}$ is actually not possible
because, as shown in the proof of Proposition \ref{pr:OmegaForwardInvariant},
the forward orbit of $Z$ is constrained to lie in $\Psi_L \cup \Psi_R$.

Also in the proof of Proposition \ref{pr:OmegaForwardInvariant}
we showed that $F(\Omega)$ is contained in the polygon with vertices \eqref{eq:FOmega}.
By \eqref{eq:Ucond}--\eqref{eq:FTcond}, this polygon does not intersect
the line through $S$ and $f(S)$ nor the line through $T$ and $f(T)$.
Thus $F(\Omega) \subset {\rm int}(\Omega)$ in the cylindrical topology.
\end{proof}

\section{Invariant expanding cones imply positive Lyapunov exponents}
\label{sec:Lyap}
\setcounter{equation}{0}

The induced map $F$ is piecewise-linear
and its switching manifolds are preimages of $\Sigma$ under $f$.
Let $\Xi \subset \Omega$ be the collection of all switching manifolds of $F$.
Then for any $Z \in \Omega \setminus \Xi$
there exists a neighbourhood $N$ of $Z$ such that
$F|_{N \cap \Omega} = f_R^q \circ f_L^p$
where $p$ and $q$ are as given in Lemma \ref{le:Fpq}.
In this neighbourhood $F$ is affine with $\rD F(Z) = A_R^q A_L^p$.
Indeed, $p$ and $q$ are constant
in each of the connected components of $\Omega \setminus \Xi$. 
Below $\| \cdot \|$ denotes the Euclidean norm on $\mathbb{R}^2$.

\begin{definition}
A nonempty set $C \subseteq \mathbb{R}^2$
is called a {\em cone} if $\alpha v \in C$ for all $v \in C$ and all $\alpha \in \mathbb{R}$.
For the map $F$ on $\Omega$, a cone $C$ is
\begin{enumerate}
\renewcommand{\labelenumi}{\roman{enumi})}
\item
{\em invariant} if $\rD F(Z) v \in C$ for all $v \in C$ and all $Z \in \Omega \setminus \Xi$,
\item
{\em contracting-invariant}
if $\rD F(Z) v \in {\rm int}(C) \cup \{ \bO \}$ for all $v \in C$ and all $Z \in \Omega \setminus \Xi$, and
\item
{\em expanding} if there exists $c > 1$ (an expansion factor) such that
$\| \rD F(Z) v \| \ge c \| v \|$ for all $v \in C$ and all $Z \in \Omega \setminus \Xi$.
\end{enumerate}
\label{df:iec}
\end{definition}

In this section we prove a general result on invariant expanding cones.
Later in \S\ref{sec:iecExistence} we obtain conditions for the existence of a contracting-invariant cone.
The result here can be applied to that cone because every contracting-invariant cone is also invariant.
The advantage of a contracting-invariant cone, over one that is only invariant,
is that it is robust to perturbations in $f$
(a detailed description and study of this is beyond the scope of this paper).

Assume $F(\Omega) \subseteq \Omega$ and let
\begin{equation}
\Xi_\infty = \bigcup_{i \ge 0} \left\{ Z \in \Omega \,\big|\, F^i(Z) \in \Xi \right\}.
\label{eq:Xiinfty}
\end{equation}
Note that $\Xi_\infty$ has zero Lebesgue measure.
For any $Z \in \Omega \setminus \Xi_\infty$, the Jacobian matrix $\rD F^n(Z)$ is well-defined for all $n \ge 1$.
The {\em Lyapunov exponent} for $Z$ in a direction $v \in \mathbb{R}^2$ is
\begin{equation}
\lambda(Z;v) = \lim_{n \to \infty} \frac{1}{n} \,\ln \left( \left\| \rD F^n(Z) v \right\| \right),
\label{eq:lyapExp}
\end{equation}
assuming this limit exists (sometimes the supremum limit is taken to avoid this issue).
The Lyapunov exponent characterises the asymptotic rate of separation
of the forward orbits of arbitrarily close points $Z$ and $Z + \delta v$.
A positive Lyapunov exponent for a bounded orbit is a commonly used indicator of chaos.
The following result shows that if there exists an invariant expanding cone,
then almost every $Z \in \Omega$ has a positive Lyapunov exponent
when the limit in \eqref{eq:lyapExp} exists.

\begin{proposition}
Suppose $\Omega \subseteq \Phi_0$ and $F(\Omega) \subseteq \Omega$.
Suppose there exists an invariant expanding cone $C \subseteq \mathbb{R}^2$ for $F$ on $\Omega$,
and let $c > 1$ be a corresponding expansion factor.
For all $v \in C$ and $Z \in \Omega \setminus \Xi_\infty$,
\begin{equation}
\liminf_{n \to \infty} \frac{1}{n} \,\ln \left( \left\| \rD F^n(Z) v \right\| \right) \ge \ln(c).
\label{eq:liminf}
\end{equation}
\label{pr:lyapExp}
\end{proposition}

\begin{proof}
Let $v_0 = v$ and for all $n \ge 1$ let $v_n = \rD F^n(Z) v$.
Then
\begin{equation}
v_n = \rD F \left( F^{n-1}(Z) \right) v_{n-1} \,,
\label{eq:vn}
\end{equation}
for all $n \ge 1$.
Since $C$ is invariant for $F$ on $\Omega$, and $F^{n-1}(Z) \in \Omega$ for all $n \ge 1$,
\eqref{eq:vn} implies $v_n \in C$ for all $n \ge 1$.
Moreover, $\| v_n \| \ge c \| v_{n-1} \|$ for each $n$.
Thus $\| v_n \| \ge c^n \| v_0 \|$ giving \eqref{eq:liminf}.
\end{proof}

\section{Dynamics of tangent vectors}
\label{sec:GH}
\setcounter{equation}{0}

Given $p, q \ge 1$, let $M_{p,q} = A_R^q A_L^p$.
Notice $\det(M_{p,q}) =\delta_L^p\delta_R^q> 0$ by (\ref{eq:params}).
In order to construct an invariant expanding cone
we study how vectors $v \in \mathbb{R}^2$ are transformed under multiplication by $M_{p,q}$.
To achieve this it suffices to consider unit vectors $(\cos(\theta), \sin(\theta))$, where $\theta \in [0,\pi)$,
because $v \mapsto M_{p,q} v$ is a linear map
and every $v \in \mathbb{R}^2$ is a scalar multiple of some such unit vector.

We endow $K = [0,\pi)$ with a cylindrical topology by identifying the numbers $0$ and $\pi$.
For any $\theta_0, \theta_1 \in K$ the closed interval $[\theta_0,\theta_1]$ is defined as
\begin{equation}
[\theta_0,\theta_1] = \begin{cases}
\left\{ \theta \in K \,\big|\, \theta_0 \le \theta \le \theta_1 \right\}, & \theta_0 \le \theta_1 \,, \\
\left\{ \theta \in K \,\big|\, \theta \le \theta_1 \right\} \cup
\left\{ \theta \in K \,\big|\, \theta \ge \theta_0 \right\}, & \theta_0 > \theta_1 \,.
\end{cases}
\label{eq:interval}
\end{equation}
Notice $\tan : K \to \mathbb{R} \cup \{ \infty \}$ is a bijection, thus $\tan^{-1}$ is unambiguous.
Here we have chosen to characterise vectors by their angle (or argument) $\theta$
because this is well-defined for all $v \ne \bO$.
In \cite{Si20e} we were able to characterise vectors with their slope $m = \tan(\theta)$,
and this made calculations regarding $v \mapsto M_{p,q} v$ somewhat simpler,
because in that setting a vector corresponding to $\theta = \frac{\pi}{2}$ (i.e.~infinite slope)
could not belong to an invariant expanding cone.

Multiplication by $M_{p,q}$ induces an `angle map' $G_{p,q} : K \to K$.
More precisely, $G_{p,q}(\theta)$ is the angle of the vector $M_{p,q} v$, where $v = (\cos(\theta), \sin(\theta))$.
We also define $H_{p,q} : K \to \mathbb{R}$ by $H_{p,q}(\theta) = \| M_{p,q} v \|$.
An example of these is shown in Fig.~\ref{fig:GH}.

\begin{figure}[b!]
\begin{center}
\includegraphics[height=9cm]{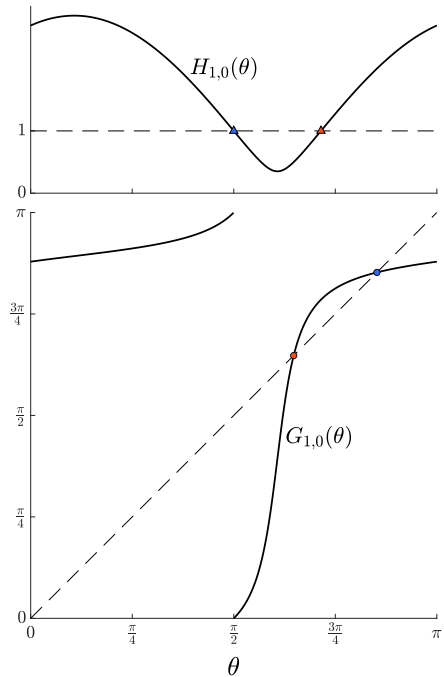}
\caption{
The angle map \eqref{eq:G} (lower plot)
and norm function \eqref{eq:H} (upper plot)
for $M_{1,0} = A_L$ with $\tau_L = 2.5$ and $\delta_L = 1$.
\label{fig:GH}
}
\end{center}
\end{figure}

In the remainder of this section we show how $G_{p,q}$ and $H_{p,q}$ can be used to establish the properties of a cone.
Given an interval $J \subset K$ of the form \eqref{eq:interval}, define the cone
\begin{equation}
C_J = \left\{ \alpha \big( \cos(\theta), \sin(\theta) \big) \,\big|\,
\alpha \in \mathbb{R}, \theta \in J \right\}.
\label{eq:CJ}
\end{equation}

\begin{lemma}
Suppose $\Omega$ satisfies the conditions of Proposition \ref{pr:OmegaForwardInvariant} and let
\begin{equation}
\Gamma = \left\{ (p,q) \,\big|\, p_{\rm min} \le p \le p_{\rm max}, q_{\rm min} \le q \le q_{\rm max} \right\}.
\label{eq:allpandq}
\end{equation}
For $F$ on $\Omega$, the cone $C_J$ is
\begin{enumerate}
\renewcommand{\labelenumi}{\roman{enumi})}
\item
invariant if $G_{p,q}(J) \subseteq J$ for all $(p,q) \in \Gamma$,
\item
contracting-invariant if $G_{p,q}(J) \subset {\rm int}(J)$ for all $(p,q) \in \Gamma$, and
\item
expanding if $H_{p,q}(\theta) > 1$ for all $\theta \in J$ and all $(p,q) \in \Gamma$.
\end{enumerate}
\label{le:CJ}
\end{lemma}

\begin{proof}
Choose any $Z \in \Omega \setminus \Xi$;
then by Proposition \ref{pr:OmegaForwardInvariant}, $\rD F(Z) = A_R^q A_L^p$ for some $(p,q) \in \Gamma$.
Choose any $v \in C_J$;
then $v = \alpha \left( \cos(\theta), \sin(\theta) \right)$ for some $\alpha \in \mathbb{R}$ and $\theta \in J$.
Then $\rD F(Z) v = \beta \left( \cos(\phi), \sin(\phi) \right)$
where $\phi = G_{p,q}(\theta)$ and $\beta = H_{p,q}(\theta)$.
For (i) we have $\phi \in J$ so $\rD F(Z) v \in C_J$
and thus $C_J$ is invariant. 
For (ii) we have $\phi \in {\rm int}(J)$ so $\rD F(Z) v \in {\rm int}(C_J) \cup \{ \bO \}$
and thus $C_J$ is contracting-invariant.
Finally for (iii) we have $\beta > 1$ so $\| \rD F(Z) v \| > 1$
and thus $C_J$ is expanding.
\end{proof}

\section{Properties of $G_{p,q}$}
\label{sec:G}
\setcounter{equation}{0}

By writing
\begin{equation}
M_{p,q} = A_R^q A_L^p = \begin{bmatrix} a & b \\ c & d \end{bmatrix},
\label{eq:M}
\end{equation}
we have
\begin{align}
\tan \left( G_{p,q}(\theta) \right) &= \frac{c \cos(\theta) + d \sin(\theta)}{a \cos(\theta) + b \sin(\theta)},
\label{eq:G} \\
H_{p,q}(\theta) &=
\sqrt{\left( a \cos(\theta) + b \sin(\theta) \right)^2 +
\left( c \cos(\theta) + d \sin(\theta) \right)^2}.
\label{eq:H}
\end{align}
We provide the following result without proof.

\begin{lemma}
The map $G_{p,q}$ is a degree-one circle map on $K$ with
\begin{equation}
\frac{d G_{p,q}}{d \theta} = \frac{\det(M_{p,q})}{(H_{p,q}(\theta))^2}.
\label{eq:dGdtheta}
\end{equation}
\label{le:G1}
\end{lemma}

That $G_{p,q}$ is a degree-one circle map is clear from the way it is defined (recall $\det(M_{p,q}) > 0$),
while \eqref{eq:dGdtheta} can be obtained directly from \eqref{eq:G} and \eqref{eq:H}.

From \eqref{eq:G} we see that fixed points of $G_{p,q}$ satisfy
\begin{equation}
b \tan^2(\theta) + (a-d) \tan(\theta) - c = 0.
\label{eq:Gfps}
\end{equation}
Note that $\theta$ is a fixed point of $G_{p,q}$
if and only if $(\cos(\theta), \sin(\theta))$ is an eigenvector of $M_{p,q}$.

\begin{lemma}
If
\begin{equation}
\det(M_{p,q}) < \tfrac{1}{4} \,{\rm trace}(M_{p,q})^2,
\label{eq:Mcondition}
\end{equation}
then $G_{p,q}$ has exactly two fixed points.
At one fixed point, $\theta_{p,q}^s$, we have $\frac{d G_{p,q}}{d \theta} = \eta$,
for some $\eta \in (0,1)$, while at the other fixed point, $\theta_{p,q}^u$,
we have $\frac{d G_{p,q}}{d \theta} = \frac{1}{\eta}$.
\label{le:G2}
\end{lemma}

Note that $\theta_{p,q}^s$ is a stable fixed point of $G_{p,q}$
while $\theta_{p,q}^u$ is an unstable fixed point of $G_{p,q}$.

\begin{proof}
By \eqref{eq:Mcondition}, $M_{p,q}$ has eigenvalues $\lambda_1, \lambda_2 \in \mathbb{R}$ with $|\lambda_1| > |\lambda_2|$.
First suppose $b \ne 0$.
Then with $\tan(\theta) = \frac{\lambda - a}{b}$ the fixed point equation \eqref{eq:Gfps}
reduces to the characteristic equation $\det \left( \lambda I - M_{p,q} \right) = 0$.
Thus $\theta_{p,q}^s = \tan^{-1} \left( \frac{\lambda_1 - a}{b} \right)$
and $\theta_{p,q}^u = \tan^{-1} \left( \frac{\lambda_2 - a}{b} \right)$
are fixed points of $G_{p,q}$.
Since \eqref{eq:Gfps} is quadratic in $\tan(\theta)$ and $\tan^{-1} : \mathbb{R} \to K$ is one-to-one,
these are the only fixed points of $G_{p,q}$.
From \eqref{eq:G}, \eqref{eq:H}, and \eqref{eq:dGdtheta} we obtain
\begin{equation}
\frac{d G_{p,q}}{d \theta} = \frac{\det(M_{p,q}) \left( 1 + \tan^2(\theta) \right)}
{\left( a + b \tan(\theta) \right)^2 \left( 1 + \tan^2 \left( G_{p,q}(\theta) \right) \right)}.
\label{eq:dGdtheta2}
\end{equation}
By evaluating \eqref{eq:dGdtheta2} at $\theta_{p,q}^s$ we obtain
$\frac{d G_{p,q}}{d \theta} \left( \theta_{p,q}^s \right) = \frac{\lambda_2}{\lambda_1}$,
where we have also substituted $\det(M_{p,q}) = \lambda_1 \lambda_2$.
Thus $\eta = \frac{\lambda_2}{\lambda_1}$ and indeed $\eta \in (0,1)$.
Similarly $\frac{d G_{p,q}}{d \theta} \left( \theta_{p,q}^u \right) = \frac{\lambda_1}{\lambda_2} = \frac{1}{\eta}$.

Finally if $b = 0$ then $a = \lambda_1$ and $d = \lambda_2$, or vice-versa.
By \eqref{eq:Gfps} the fixed points of $G_{p,q}$ are $\frac{\pi}{2}$ and $\tan^{-1} \big( \frac{c}{a - d} \big)$
and again from \eqref{eq:dGdtheta2} we see that at these points $\frac{d G_{p,q}}{d \theta}$
has the values $\frac{\lambda_2}{\lambda_1}$ and $\frac{\lambda_1}{\lambda_2}$.
\end{proof}

\section{Existence of an invariant expanding cone}
\label{sec:iecExistence}
\setcounter{equation}{0}

In this section we use the stable fixed points of the angle maps $G_{p,q}$ to construct an interval $J$
for which, if certain conditions are satisfied, the corresponding cone $C_J$ is invariant.
We also show that $J$ can be enlarged slightly to obtain a cone that is contracting-invariant.
Finally we characterise the expansion condition $H_{p,q}(\theta) > 1$.

Let $1 \le p_{\rm min} < p_{\rm max}$ and $1 \le q_{\rm min} < q_{\rm max}$ be given and
suppose \eqref{eq:Mcondition} is satisfied for all $(p,q) \in \Gamma$.

\begin{definition}
If there exists a closed interval $J \subset K$
such that $\theta_{p,q}^s \in J$ and $\theta_{p,q}^u \notin J$
for all $(p,q) \in \Gamma$,
then we say that the fixed points of the $G_{p,q}$ are {\em unmixed}.
\label{df:unmixed}
\end{definition}

If the fixed points are unmixed, then the smallest such $J$ (really the intersection of all such $J$)
is an interval with stable fixed points as its endpoints.
That is,
\begin{equation}
J = \left[ \theta_{p_1,q_1}^s, \theta_{p_2,q_2}^s \right],
\label{eq:J}
\end{equation}
for some $(p_1,q_1), (p_2,q_2) \in \Gamma$.
Fig.~\ref{fig:GHmany} shows an example.

\begin{figure}[b!]
\begin{center}
\includegraphics[height=9cm]{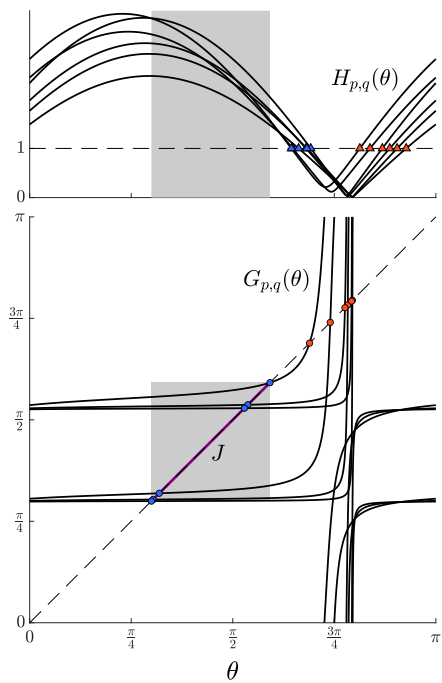}
\caption{
An example of the angle maps $G_{p,q}$ (lower plot)
and the norm functions $H_{p,q}$ (upper plot) for all $(p,q) \in \Gamma$.
This figure was generated using $(\tau_L,\delta_L,\tau_R,\delta_R) = (1,0.2,-1.2,2)$ and
$p_{\rm min} = 1$, $p_{\rm max} = 3$, $q_{\rm min} = 1$, and $q_{\rm max} = 2$.
In this case the fixed points of the $G_{p,q}$ are unmixed
and the interval $J$, equation \eqref{eq:J}, is shown in the lower plot.
From the upper plot notice that $H_{p,q}(\theta) > 1$ for all $\theta \in J$ and all $(p,q) \in \Gamma$
as required for $C_J$ to be expanding (see Lemma \ref{le:CJ}).
\label{fig:GHmany}
}
\end{center}
\end{figure}

\begin{proposition}
Suppose \eqref{eq:Mcondition} is satisfied for all $(p,q) \in \Gamma$ and the fixed points of the $G_{p,q}$ are unmixed.
With $J$ given by \eqref{eq:J}, $G_{p,q}(J) \subseteq J$ for all $(p,q) \in \Gamma$.
Moreover there exists $\ee > 0$ such that for
$J_\ee = \left[ \theta_{p_1,q_1}^s - \ee, \theta_{p_2,q_2}^s + \ee \right]$ we have
$G_{p,q}(J_\ee) \subset {\rm int}(J_\ee)$ for all $(p,q) \in \Gamma$.
\label{pr:J}
\end{proposition}

\begin{proof}
Choose any $(p,q) \in \Gamma$.
Then $\theta_{p,q}^s \in J$.
Since $\frac{d G_{p,q}}{d \theta} \left( \theta_{p,q}^s \right) \in (0,1)$ (by Lemma \ref{le:G2})
for any sufficiently small interval $I$ containing $\theta_{p,q}^s$ we have $G_{p,q}(I) \subseteq I$.
Since $\frac{d G_{p,q}}{d \theta} > 0$ on $K$ (by Lemma \ref{le:G1})
$G_{p,q}(I) \subseteq I$ remains true when $I$ is enlarged to any interval containing no other fixed points of $G_{p,q}$.
Since the fixed points are unmixed, $J$ is such an interval.
That is, $G_{p,q}(J) \subseteq J$.

If $\theta_{p,q}^s = \theta_{p_1,q_1}^s$
then $\frac{d G_{p,q}}{d \theta} \left( \theta_{p,q}^s \right) \in (0,1)$ implies there exists $\ee_{p,q}^- > 0$ such that
\begin{equation}
G_{p,q} \left( \left[ \theta_{p_1,q_1}^s - \ee, \theta_{p_1,q_1}^s \right] \right) \subset
{\rm int} \left( \left[ \theta_{p_1,q_1}^s - \ee, \theta_{p_2,q_2}^s \right] \right),
\label{eq:eeMinus}
\end{equation}
for all $\ee \in \left( 0, \ee_{p,q}^- \right]$.
If instead $\theta_{p,q}^s \ne \theta_{p_1,q_1}^s$
then $G_{p,q} \left( \theta_{p_1,q_1}^s \right) \in {\rm int}(J)$ so again there exists $\ee_{p,q}^- > 0$ such that
\eqref{eq:eeMinus} holds for all $\ee \in \left( 0, \ee_{p,q}^- \right]$.
Similarly there exists $\ee_{p,q}^+ > 0$ such that 
\begin{equation}
G_{p,q} \left( \left[ \theta_{p_2,q_2}^s, \theta_{p_2,q_2}^s + \ee \right] \right) \subset
{\rm int} \left( \left[ \theta_{p_1,q_1}^s, \theta_{p_2,q_2}^s + \ee \right] \right),
\label{eq:eePlus}
\end{equation}
for all $\ee \in \left( 0, \ee_{p,q}^+ \right]$.
Then $G_{p,q}(J_\ee) \subset {\rm int}(J_\ee)$ for all $(p,q) \in \Gamma$ with
$\ee = \underset{(p,q) \in \Gamma}{\min} \left( \min \left[ \ee_{p,q}^-, \ee_{p,q}^+ \right] \right)$.
\end{proof}

In order for $C_J$ to be expanding
we require $H_{p,q}(\theta) > 1$ for all $\theta \in J$ and all $(p,q) \in \Gamma$.
Solutions to $H_{p,q}(\theta) = 1$ satisfy
\begin{equation}
\left( b^2 + d^2 - 1 \right) \tan^2(\theta) + 2 (a b + c d) \tan(\theta) + \left( a^2 + c^2 - 1 \right) = 0,
\label{eq:Hsolns}
\end{equation}
so if $b^2 + d^2 - 1 \ne 0$, as is generically the case,
then \eqref{eq:Hsolns} is quadratic in $\tan(\theta)$
and consequently $H_{p,q}(\theta) = 1$ has at most two solutions on $K$.
When $H_{p,q}(\theta) = 1$ has exactly two solutions on $K$,
$H_{p,q}(\theta)$ is decreasing at one solution, call it $\theta_{p,q}^{\rm dec}$,
and increasing at the other solution, call it $\theta_{p,q}^{\rm inc}$.
Then $H_{p,q}(\theta) \le 1$ if and only if $\theta \in \left[ \theta_{p,q}^{\rm dec}, \theta_{p,q}^{\rm inc} \right]$
and so we have the following result.

\begin{lemma}
Suppose $H_{p,q}(\theta) = 1$ has exactly two solutions.
Then $H_{p,q}(\theta) > 1$ for all $\theta \in J$ if and only if
\begin{equation}
J \cap \left[ \theta_{p,q}^{\rm dec}, \theta_{p,q}^{\rm inc} \right] = \varnothing.
\label{eq:noIntersection}
\end{equation}
\label{le:H}
\end{lemma}

\section{Large values of $p_{\rm min}$}
\label{sec:pMin}
\setcounter{equation}{0}

Here we impose the additional constraint
\begin{equation}
\tau_L > \delta_L + 1,
\label{eq:extraCondition}
\end{equation}
and show that if $p_{\rm min}$ is sufficiently large then we can expect the fixed points of the $G_{p,q}$ to be unmixed
and the cone $C_J$, with $J$ as in Proposition \ref{pr:J}, to be invariant and expanding.
Condition \eqref{eq:extraCondition} implies $f$ has a saddle fixed point in $Q_2$ with positive eigenvalues.
For large values of $p$ the effect of the saddle dominates the nature of $M_{p,q}$ 
in a way that is favourable for $C_J$ for be invariant and expanding.

First observe that for any $p,q \ge 1$ we can write
\begin{equation}
G_{p,q}(\theta) = G_{0,q} \left( G_{1,0}^p(\theta) \right).
\label{eq:GAsComposition}
\end{equation}
In \eqref{eq:GAsComposition} we apply the map $G_{1,0}$ $p$ times, then apply $G_{0,q}$.
Since $M_{1,0} = A_L$, by Lemma \ref{le:G2} the condition \eqref{eq:extraCondition}
ensures $G_{1,0}$ has unique stable and unstable fixed points
$\theta_{1,0}^s$ and $\theta_{1,0}^u$ 
(shown in Fig.~\ref{fig:GH}).

\begin{proposition}
Let $q \ge 1$.
Suppose \eqref{eq:extraCondition} is satisfied and $G_{0,q} \left( \theta_{1,0}^s \right) \ne \theta_{1,0}^u$.
Then \eqref{eq:Mcondition} is satisfied for all sufficiently large values of $p \ge 1$
and $\theta_{p,q}^s \to G_{0,q} \left( \theta_{1,0}^s \right)$
and $\theta_{p,q}^u \to \theta_{1,0}^u$ as $p \to \infty$.
Also $H_{p,q}(\theta) \to \infty$ as $p \to \infty$ for all $\theta \ne \theta_{1,0}^u$.
\label{pr:limit}
\end{proposition}

\begin{proof}
The basin of attraction of the stable fixed point $\theta_{1,0}^s$ of $G_{1,0}$
consists of all $\theta \in K$ except the unstable fixed point $\theta_{1,0}^u$
(because $v \mapsto A_L v$ is linear).
Thus $G_{1,0}^p(\theta) \to \theta_{1,0}^s$ as $p \to \infty$ for all $\theta \ne \theta_{1,0}^u$.
Thus by \eqref{eq:GAsComposition},
$G_{p,q}(\theta) \to G_{0,q} \left( \theta_{1,0}^s \right)$ as $p \to \infty$ for all $\theta \ne \theta_{1,0}^u$.
But $G_{p,q}$ is a degree-one circle map (Lemma \ref{le:G1}) and a diffeomorphism
thus, for sufficiently large values of $p$, must have a stable fixed point $\theta_{p,q}^s$
with $\theta_{p,q}^s \to G_{0,q} \left( \theta_{1,0}^s \right)$ as $p \to \infty$
and an unstable fixed point $\theta_{p,q}^u$
with $\theta_{p,q}^u \to \theta_{1,0}^u$ as $p \to \infty$.
The existence of these fixed points implies \eqref{eq:Mcondition}.

By \eqref{eq:extraCondition}, $A_L$ has eigenvalues $0 < \lambda_2 < 1 < \lambda_1$.
The stable fixed point $\theta_{1,0}^s$ of $G_{1,0}$ corresponds to the unstable eigen-direction of $A_L$,
thus $H_{1,0} \left( \theta_{1,0}^s \right) = \lambda_1$.
For any $\theta \ne \theta_{1,0}^u$, $G_{1,0}^p(\theta) \to \theta_{1,0}^s$ implies
$H_{p,0}(\theta) \sim \lambda_1^p$,
and so $H_{p,q}(\theta) \to \infty$ because $\det(A_R) = \delta_R \ne 0$.
\end{proof}

\section{A computer-assisted proof of chaos}
\label{sec:algorithm}
\setcounter{equation}{0}

\begin{algorithm}
For a given map of the form \eqref{eq:f} with \eqref{eq:params} we perform the following steps.
\begin{enumerate}
\item
Let $p_{\rm max}$ be the smallest $p \in \left\{ 2,\ldots,\min[p^*,15] \right\}$
for which \eqref{eq:Ucond} and \eqref{eq:FScond} are satisfied with $S = \left( 0, c_p \right)$; if $p_{\rm max}$ does not exist {\sc stop}.
\item
Let $p_{\rm min}$ be the largest $p \in \left\{ 1,\ldots,p_{\rm max}-1 \right\}$
for which \eqref{eq:Vcond} and \eqref{eq:FTcond} are satisfied with $T = \left( 0, c_p \right)$; if $p_{\rm min}$ does not exist {\sc stop}.
\item
Let $\Gamma$ be given by \eqref{eq:allpandq} using \eqref{eq:qMinqMax}.
If \eqref{eq:Mcondition} is not satisfied for some $(p,q) \in \Gamma$ then {\sc stop}.
\item
If the fixed points of the $G_{p,q}$ are mixed (i.e.~not unmixed) then {\sc stop}.
\item
Let $J$ be the interval in Proposition \ref{pr:J}.
If $\theta_{p,q}^{\rm dec}$ and $\theta_{p,q}^{\rm inc}$ do not exist
or $J \cap \left[ \theta_{p,q}^{\rm dec}, \theta_{p,q}^{\rm inc} \right] \ne \varnothing$
for some $(p,q) \in \Gamma$, then {\sc stop}, otherwise output {\sc chaos}.
\end{enumerate}
\label{al:theAlgorithm}
\end{algorithm}

The upper bound of $p = 15$ imposed in Step 1 is a suitable
finite bound to ensure the number of computations is finite.
Algorithm \ref{al:theAlgorithm} picks the largest and smallest allowed values of $p_{\rm min}$ and $p_{\rm max}$, respectively,
in order to minimise the size of the set $\Gamma$ and so minimise the number of conditions in Steps 3 -- 5 that need to hold.

\begin{theorem}
Suppose Algorithm \ref{al:theAlgorithm} outputs {\sc chaos}.
With $\Omega$ as given in Proposition \ref{pr:OmegaForwardInvariant}
and $J$ as given in Proposition \ref{pr:J},
\begin{equation}
\liminf_{n \to \infty} \frac{1}{n} \,\ln \left( \left\| \rD f^n(Z) v \right\| \right) > 0,
\label{eq:liminf2}
\end{equation}
for all $Z \in \Omega \setminus \Xi_\infty$ and all $v \in C_J$.
\label{th:main}
\end{theorem}

\begin{proof}
Since Algorithm \ref{al:theAlgorithm} does not stop in Steps 1 and 2,
the assumptions in Proposition \ref{pr:OmegaForwardInvariant} hold so $\Omega \subseteq \Phi_0$ and $F(\Omega) \subseteq \Omega$.
Since Algorithm \ref{al:theAlgorithm} does not stop in Steps 3 and 4,
the assumptions in Proposition \ref{pr:J} hold so $C_J$ is invariant by Lemma \ref{le:CJ}(i).
Since Algorithm \ref{al:theAlgorithm} does not stop in Step 5,
\eqref{eq:noIntersection} holds for all $(p,q) \in \Gamma$
thus $C_J$ is expanding by Lemma \ref{le:CJ}(iii) and Lemma \ref{le:H}.
Let $c > 1$ be an expansion factor for $C_J$.
Then \eqref{eq:liminf} holds by Proposition \ref{pr:lyapExp}.
This implies the left hand-side of \eqref{eq:liminf2} is at least
$\frac{\ln(c)}{p_{\rm max} + q_{\rm max}} > 0$.
\end{proof}

\section{Numerical results}
\label{sec:numerics}
\setcounter{equation}{0}

\begin{figure}[b!]
\begin{center}
\setlength{\unitlength}{1cm}
\begin{picture}(16,8)
\put(0,0){\includegraphics[height=8cm]{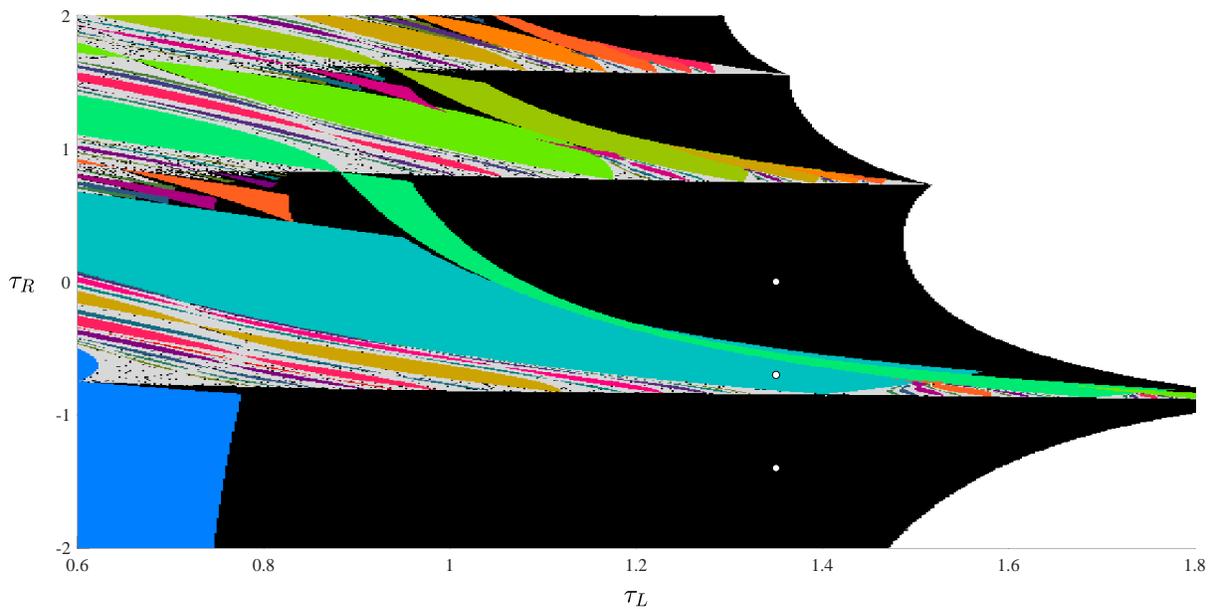}}
\end{picture}
\caption{
A numerically obtained bifurcation diagram of \eqref{eq:f} with \eqref{eq:dLdR}.
The map has a periodic attractor with period at most $30$ in the coloured regions,
some other attractor with negative Lyapunov exponent in the grey regions,
an attractor with a positive Lyapunov exponent in the black regions,
and no attractor in the white regions.
The three dots indicate the parameter values of Fig.~\ref{fig:qqThree}.
\label{fig:bifSet1}
}
\end{center}
\end{figure}

In this section we illustrate Algorithm \ref{al:theAlgorithm}
over the two-dimensional slice of the parameter space of \eqref{eq:f} defined by fixing
\begin{align}
\delta_L &= 0.2, &
\delta_R &= 2, &
\mu &= 1.
\label{eq:dLdR}
\end{align}
First, Fig.~\ref{fig:bifSet1} shows a two-parameter bifurcation diagram of \eqref{eq:f} with \eqref{eq:dLdR}.
Coloured regions are where there exists a stable periodic solution of period at most $30$.
To characterise the long-term dynamics outside these regions,
we computed the forward orbit of the origin over a grid of values of $\tau_L$ and $\tau_R$.
Grey regions are where a numerically computed Lyapunov exponent for this orbit was negative;
black regions are where this Lyapunov exponent was positive.
White regions are where the orbit appeared to diverge.

\begin{figure}[b!]
\begin{center}
\setlength{\unitlength}{1cm}
\begin{picture}(16,8)
\put(0,0){\includegraphics[height=8cm]{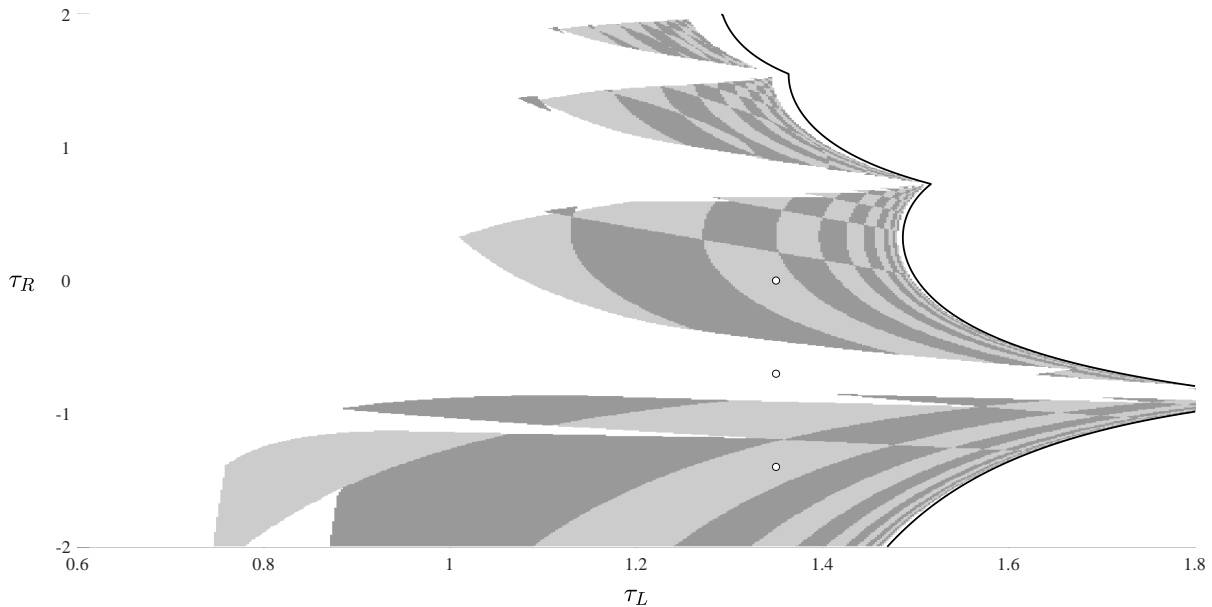}}
\end{picture}
\caption{
Regions where Algorithm \ref{al:theAlgorithm} outputs {\sc chaos}
(light grey where $p_{\rm max} - p_{\rm min}$ is even; dark grey where $p_{\rm max} - p_{\rm min}$ is odd)
for the parameter region of Fig.~\ref{fig:bifSet1}.
The black curve is a locus of homoclinic corners \cite{Si16b}.
\label{fig:bifSet2}
}
\end{center}
\end{figure}

Fig.~\ref{fig:bifSet2} illustrates the results of Algorithm \ref{al:theAlgorithm} over the same parameter range.
Shaded regions are where Algorithm \ref{al:theAlgorithm} outputted {\sc chaos}.
In order to reveal some of the underlying processes, the region is
light grey if $p_{\rm max} - p_{\rm min}$ is even
and dark grey if $p_{\rm max} - p_{\rm min}$ is odd.

As expected these regions form a proper subset of the black regions of Fig.~\ref{fig:bifSet1}.
That is, Algorithm \ref{al:theAlgorithm} outputs {\sc chaos} whenever our numerical estimation of the Lyapunov exponent is positive,
but the converse is not necessarily true.
Nevertheless, at least for the slice of parameter space shown,
Algorithm \ref{al:theAlgorithm} is quite successful in that it outputs {\sc chaos}
over the majority of the region where numerical simulations suggest a chaotic attractor exists.

\begin{figure}[t!]
\begin{center}
\setlength{\unitlength}{1cm}
\begin{picture}(10,16.8) 
\put(0,11.6){\includegraphics[height=5cm]{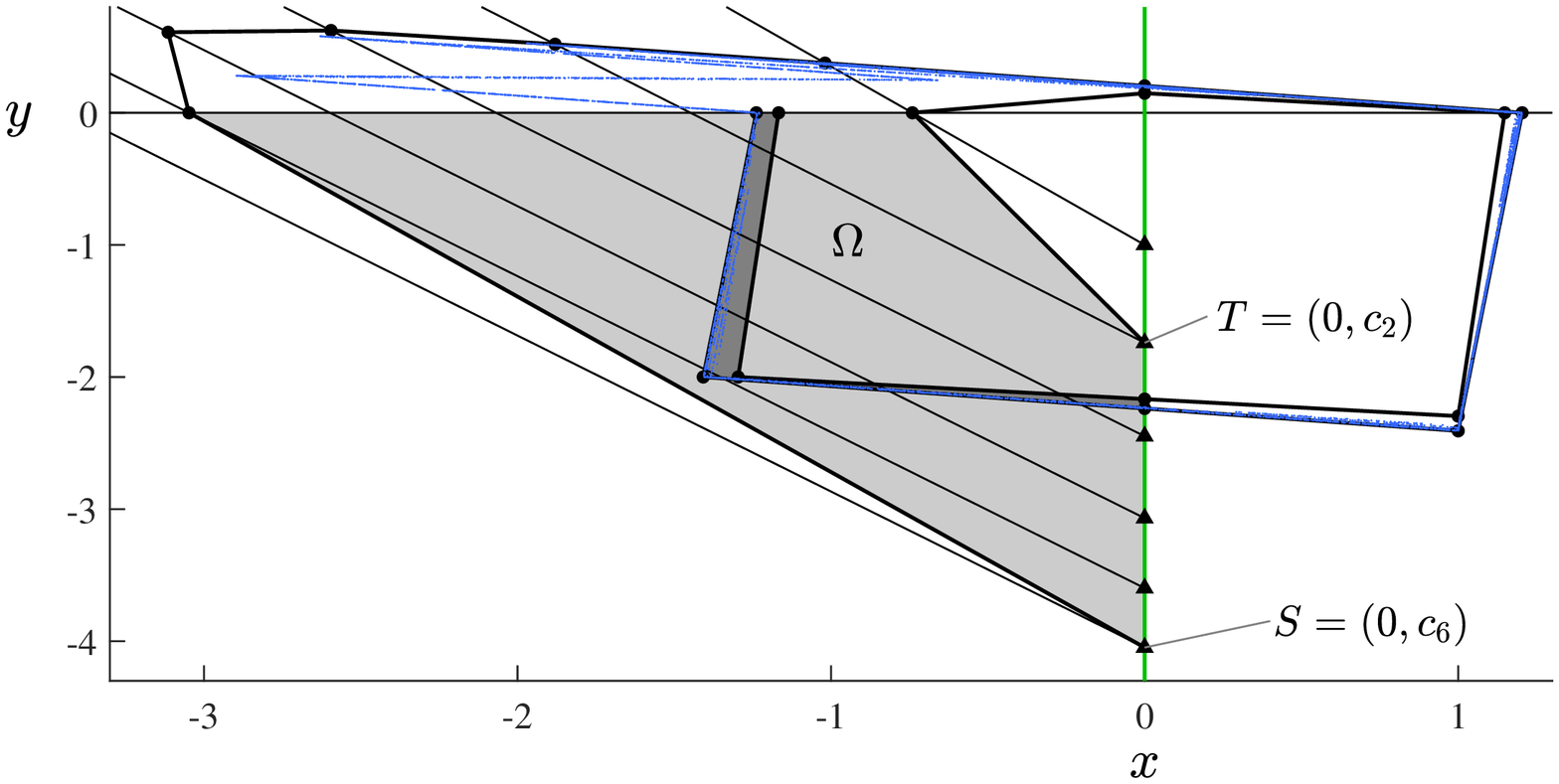}}
\put(0,5.8){\includegraphics[height=5cm]{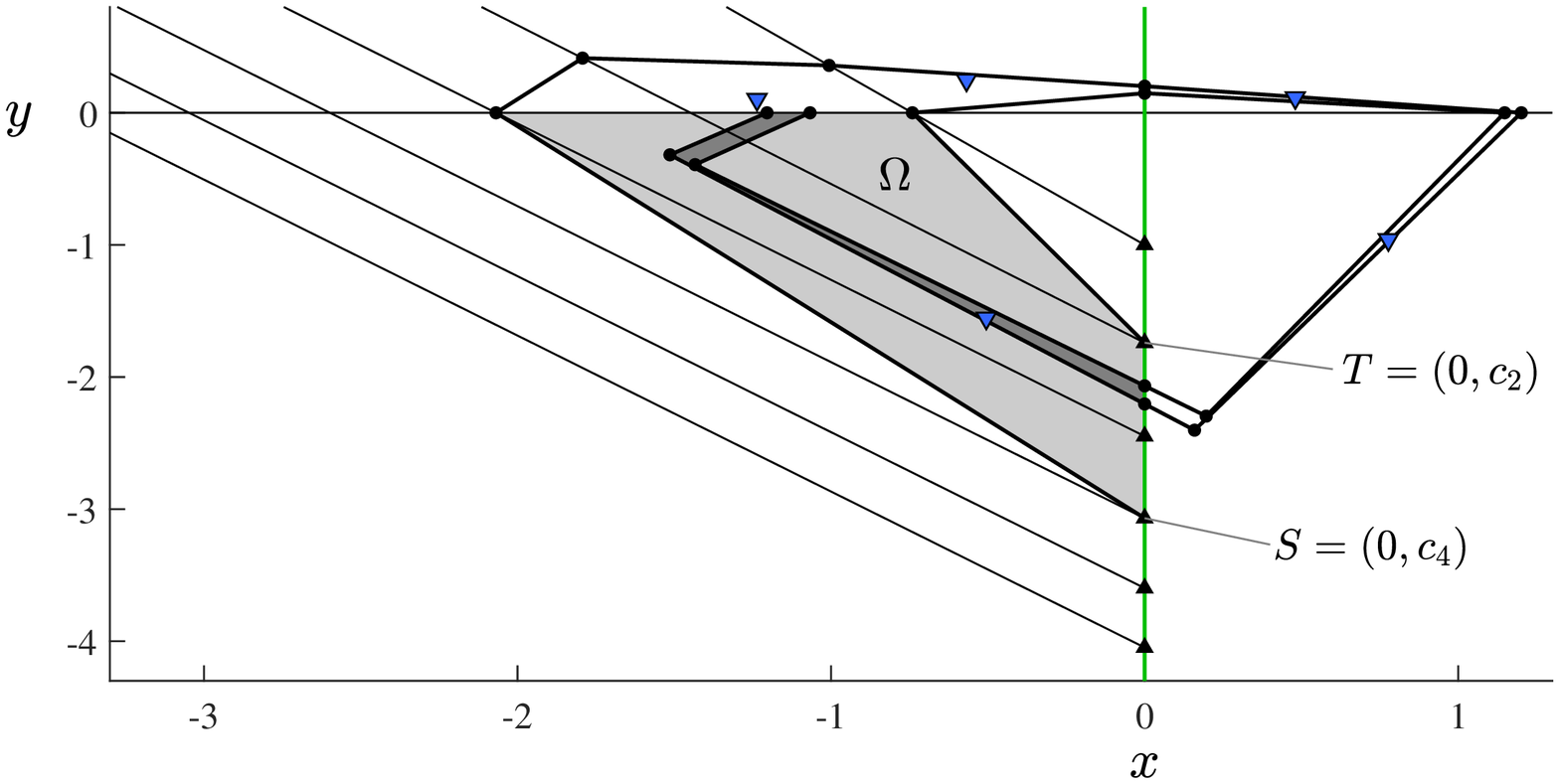}}
\put(0,0){\includegraphics[height=5cm]{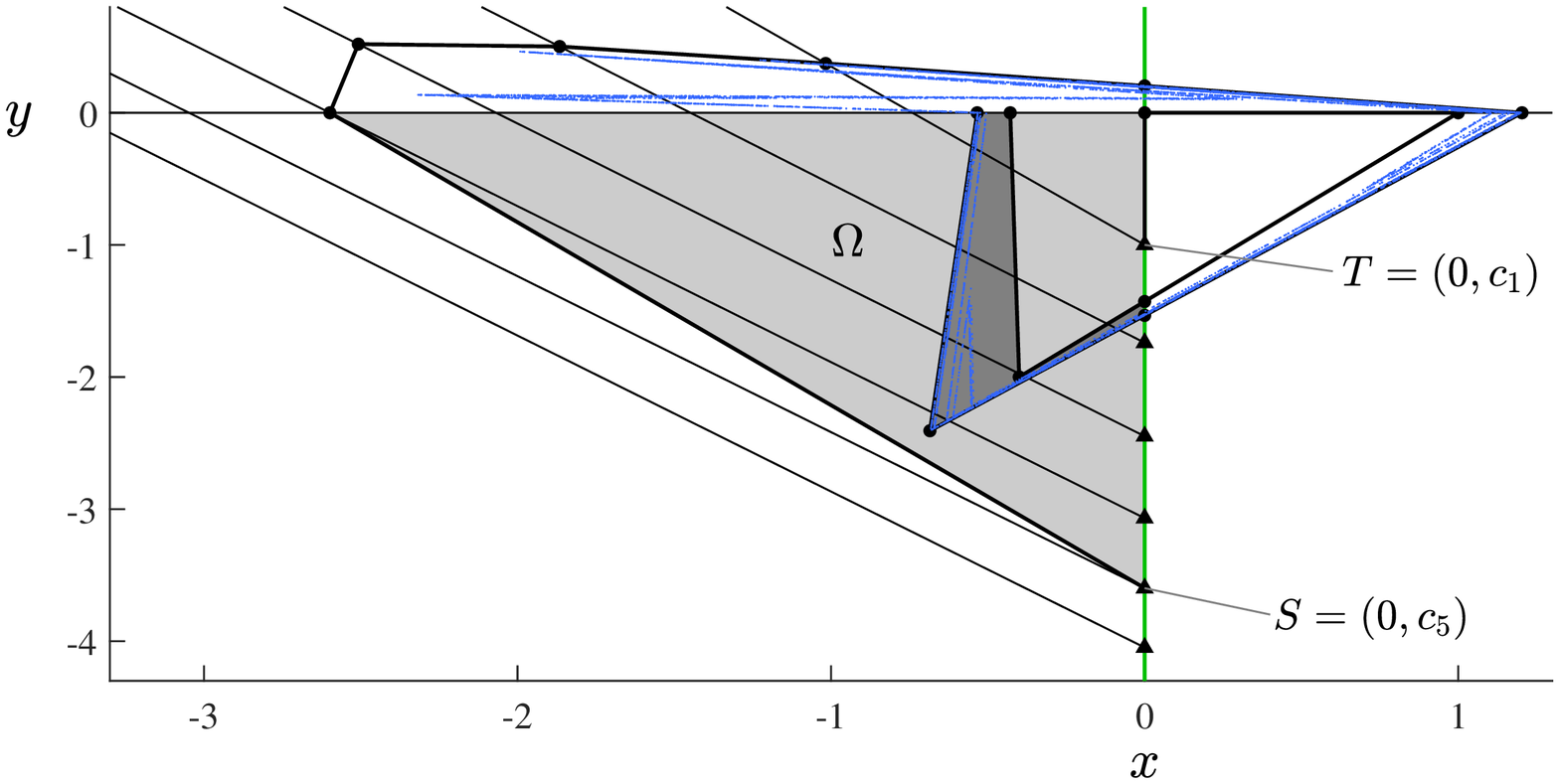}}
\put(4,16.8){\small a)~~$\tau_R = 0$}
\put(4,11){\small b)~~$\tau_R = -0.7$}
\put(4,5.2){\small c)~~$\tau_R = -1.4$}
\end{picture}
\caption{
Phase portraits of \eqref{eq:f} with \eqref{eq:dLdR}, $\tau_L = 1.35$, and three different values of $\tau_R$
corresponding to the dots in Figs.~\ref{fig:bifSet1} and \ref{fig:bifSet2}.
In each plot we show $\Omega$ (as produced by Algorithm \ref{al:theAlgorithm} and shaded light grey),
$F(\Omega)$ (dark grey), and a numerically computed attractor (blue).
\label{fig:qqThree}
}
\end{center}
\end{figure}

Phase portraits corresponding to the three sample parameter combinations
highlighted in Figs.~\ref{fig:bifSet1} and \ref{fig:bifSet2} are shown in Fig.~\ref{fig:qqThree}.
In panels (a) and (c) Algorithm \ref{al:theAlgorithm} outputted {\sc chaos}
with $(p_{\rm min},p_{\rm max}) = (2,6)$ in panel (a)
and $(p_{\rm min},p_{\rm max}) = (1,5)$ in panel (c).
For these parameter values \eqref{eq:f} appears to have a unique attractor (shown with blue dots).
In panel (b) Algorithm \ref{al:theAlgorithm} completed Steps 1 and 2,
producing $(p_{\rm min},p_{\rm max}) = (2,4)$,
but stopped at Step 3 because \eqref{eq:Mcondition} is not satisfied for $(p,q) = (3,2)$.
Indeed for these parameter values \eqref{eq:f} has an asymptotically stable period-$5$ solution (shown with blue triangles).
The point $Z \in \Omega$ of this periodic solution
satisfies $F(Z) = Z$ with $(p,q) = (3,2)$ in \eqref{eq:Fpq}.

As a final remark, for the considered parameter slice
attractors are destroyed on the piecewise-smooth curve shown in Fig.~\ref{fig:bifSet2}.
On this curve the stable and unstable manifolds of the saddle fixed point in $x < 0$ attain a `homoclinic corner' 
(a first homoclinic tangency except the invariant manifolds are piecewise-linear).
As seen in Fig.~\ref{fig:bifSet1} the periodicity regions accumulate at the kinks of this curve
and this was proved in a general setting in \cite{Si20}.

\section{Discussion}
\label{sec:conc}
\setcounter{equation}{0}

We have shown how numerical methods can be used to verify (up to numerical accuracy)
a finite set of conditions that imply a chaotic attractor exists in the 2d BCNF.
This avoids lengthy computations and estimates of limiting quantities. 
There are further embellishments that could be employed,
for example rather than check the conditions at individual points in parameter space
one could determine codimension-one surfaces in parameter space that bound where each condition holds.
If these surfaces bound an open subset of parameter space,
then in this set the 2d BCNF exhibits robust chaos. 

It remains to further relate the conditions to the dynamics of the map.
If the failure of a condition does not correspond to the destruction of a chaotic attractor
(which does occur in a similar setting in \cite{Si20e}),
it may correspond to a crisis where the attractor jumps in size (see \cite{GlSi20b})
or experiences some tangible change to its geometry.

It is natural to ask how our approach can be applied to maps that are piecewise-smooth, but not piecewise-linear.
If we do not drop the nonlinear terms used to create the 2d BCNF
we expect that these terms can be controlled by assuming $\mu$ is small and rescaling.
In this way the linear terms should dominate
and the trapping region and contracting-invariant cone should persist.
Already Young \cite{Yo85} has a theoretical analysis that shows nonlinear terms can be incorporated into the analysis in some settings.
   
The accurate simulation of long time solutions to piecewise-smooth systems can be a problematic (micro-chaos is one aspect of this \cite{GlKo10b}).
The finite time calculations required by our geometric approach provides significantly
more robustness to the use of computers for proving the existence of chaotic attractors.

\begin{section}*{Funding}
This work was supported by Marsden Fund contract MAU1809, managed by Royal Society Te Ap\={a}rangi.
\end{section}


\begin{thebibliography}{10}

\bibitem{AwLa03}
J.~Awrejcewicz and C.~Lamarque.
\newblock {\em Bifurcation and Chaos in Nonsmooth Mechanical Systems.}
\newblock World Scientific, Singapore, 2003.

\bibitem{BaYo98}
S.~Banerjee, J.A. Yorke, and C.~Grebogi.
\newblock Robust chaos.
\newblock {\em Phys. Rev. Lett.}, 80(14):3049--3052, 1998.

\bibitem{DeVa93}
W.~de~Melo and S.~van Strien.
\newblock {\em One-Dimensional Dynamics.}
\newblock Springer-Verlag, New York, 1993.

\bibitem{DiBu08}
M.~di~Bernardo, C.J. Budd, A.R. Champneys, and P.~Kowalczyk.
\newblock {\em Piecewise-smooth Dynamical Systems. Theory and Applications.}
\newblock Springer-Verlag, New York, 2008.

\bibitem{EdGl14}
R.~Edwards and L.~Glass.
\newblock Dynamics in genetic networks.
\newblock {\em Amer. Math. Monthly}, 121(9):793--809, 2014.

\bibitem{Gl16e}
P.~Glendinning.
\newblock Bifurcation from stable fixed point to {2D} attractor in the border
  collision normal form.
\newblock {\em IMA J. Appl. Math.}, 81(4):699--710, 2016.

\bibitem{Gl17}
P.~Glendinning.
\newblock Robust chaos revisited.
\newblock {\em Eur. Phys. J. Special Topics}, 226(9):1721--1738, 2017.

\bibitem{GlJe19}
P.~Glendinning and M.R. Jeffrey.
\newblock {\em An Introduction to Piecewise Smooth Dynamics.}
\newblock Birkhauser, Boston, 2019.

\bibitem{GlKo10b}
P.~Glendinning and P.~Kowalczyk.
\newblock Micro-chaotic dynamics due to digital sampling in hybrid systems of
  {F}ilippov type.
\newblock {\em Phys. D}, 239:58--71, 2010.

\bibitem{GlSi20b}
P.A. Glendinning and D.J.W. Simpson.
\newblock Robust chaos and the continuity of attractors.
\newblock {\em Trans. Math. Appl.}, 4(1):tnaa002, 2020.

\bibitem{GlSi21}
P.A. Glendinning and D.J.W. Simpson.
\newblock A constructive approach to robust chaos using invariant manifolds and
  expanding cones.
\newblock {\em Discrete Contin. Dyn. Syst.}, 41(7):3367--3387, 2021.

\bibitem{Jo03}
M.~Johansson.
\newblock {\em Piecewise Linear Control Systems.}, volume 284 of {\em Lecture
  Notes in Control and Information Sciences.}
\newblock Springer-Verlag, New York, 2003.

\bibitem{KoLi11}
L.~Kocarev and S.~Lian, editors.
\newblock {\em Chaos-Based Cryptography. Theory, Algorithms and Applications}.
\newblock Springer, New York, 2011.

\bibitem{Lo78}
R.~Lozi.
\newblock Un attracteur \'{e}trange(?) du type attracteur de {H}\'{e}non.
\newblock {\em J. Phys. (Paris)}, 39(C5):9--10, 1978.
\newblock In French.

\bibitem{Mi80}
M.~Misiurewicz.
\newblock Strange attractors for the {L}ozi mappings.
\newblock In R.G. Helleman, editor, {\em Nonlinear dynamics, Annals of the New
  York Academy of Sciences}, pages 348--358, New York, 1980. Wiley.

\bibitem{NuYo92}
H.E. Nusse and J.A. Yorke.
\newblock Border-collision bifurcations including ``period two to period
  three'' for piecewise smooth systems.
\newblock {\em Phys. D}, 57:39--57, 1992.

\bibitem{Si16}
D.J.W. Simpson.
\newblock Border-collision bifurcations in $\mathbb{R}^n$.
\newblock {\em SIAM Rev.}, 58(2):177--226, 2016.

\bibitem{Si16b}
D.J.W. Simpson.
\newblock Unfolding homoclinic connections formed by corner intersections in
  piecewise-smooth maps.
\newblock {\em Chaos}, 26:073105, 2016.

\bibitem{Si20e}
D.J.W. Simpson.
\newblock Detecting invariant expanding cones for generating word sets to
  identify chaos in piecewise-linear maps.
\newblock {\em Submitted}., 2020.

\bibitem{Si20}
D.J.W. Simpson.
\newblock Unfolding codimension-two subsumed homoclinic connections in
  two-dimensional piecewise-linear maps.
\newblock {\em Int. J. Bifurcation Chaos}, 30(3):2030006, 2020.

\bibitem{Yo85}
L.-S. Young.
\newblock Bowen-{R}uelle measures for certain piecewise hyperbolic maps.
\newblock {\em Trans. Amer. Math. Soc.}, 287(1):41--48, 1985.

\end{thebibliography}

\end{document}